\documentclass[11pt,a4paper]{article}
\usepackage[pdftex]{graphicx}
\usepackage{amsmath,amssymb,amsfonts,amsthm}
\numberwithin{equation}{section}
\usepackage{indentfirst}
\usepackage{enumitem} 
\usepackage[margin=1.3in]{geometry}
\usepackage{fancyhdr}

\usepackage[colorlinks=true,linkcolor=magenta,citecolor=blue,urlcolor=cyan]{hyperref}
\usepackage{titlesec}
\usepackage{etoolbox}
\usepackage{graphicx}  
\usepackage{changepage}
\theoremstyle{plain}
\newtheorem{theorem}{Theorem}[section]
\newtheorem{lemma}[theorem]{Lemma}
\newtheorem{corollary}[theorem]{Corollary}

\usepackage{comment}

\makeatletter
\renewcommand{\maketitle}{
	\begin{center}
		{\Large\bfseries{\@title}\par}
		\vskip 1em
		{\normalsize
			\lineskip .5em
			\begin{tabular}[t]{c}
				\@author
			\end{tabular}\par}
		\vskip 1.5em
	\end{center}
}
\makeatother

\usepackage{color}
\usepackage{xcolor}
\usepackage[normalem]{ulem} 
\usepackage{soul}

\usepackage{graphicx} 
\usepackage{xstring}  

\usepackage{xstring}
\usepackage{graphicx} 

\usepackage{marvosym}

\begin{document}
	\begin{center}{\bf \large
			{On Congruences modulo powers of 3 for Simultaneously $l$-Regular and $t$-Distinct Partitions in $k$-Colors\\[0.15cm] } }\vspace{0.2cm}
	\end{center}

	\begin{center}
		\footnotemark[1]
		\bf Anjelin Mariya Johnson, \footnotemark[2]
		\bf S. N. Fathima  
	\end{center}
	\vspace{1 cm}
	\begin{center}
		\begin{minipage}{0.85\textwidth}  
			
			\noindent{\bf Abstract:} 
			 In 2017, Keith defined and studied a multi-restricted partition function $RD^{l,t}(n)$, the number of integer partitions of n into parts that are simultaneously $l$-regular and $t$-distinct if none of the parts is divisible by $l$ and parts appear fewer than $t$ times. Thereafter, in a series of papers Buragohain et al. and Nadji et al. examined combinatorial and arithmetic properties of $RD^{l,t}(n)$. In this article, we focus on simultaneously $l$-regular and $t$-distinct partitions in 3-colors, $RD^{l,t}_3(n)$ and establish infinite families of congruences modulo powers of 3 for specific values of $(l,t)$. For instance, for positive integers $n$ and $k$, we have
			 \begin{align*}
			 	\sum_{n=0}^{\infty}RD_3^{3,3}\left(3^kn+\frac{3^k+1}{2}\right)q^n\equiv0 \pmod{3^{k+1}}.
			 \end{align*}
			 Our approach involves Hirschhorn's H operator technique.
			\\
			\vspace{.125cm}
			
			\noindent {\bf \small Keywords} : Generating Function, Partitions, Congruences, $q$-series.\\
			\vspace{.05cm}
			
			\noindent {\bf \small Mathematical Subject Classification (2020)} :  05A15, 05A17, 11P81, 11P83.
		\end{minipage}
	\end{center}
	
	\section {Introduction}
	In his 2017 paper \cite{keith}, William Keith introduced the partition function $RD^{\ell,t}$, which is simultaneously $s$-regular and $t$-distinct partition function. 
	A partition of positive integer $n$ is said to be simultaneously $l$-regular and $t$-distinct if none of the parts are divisible by $l$ and parts appear fewer than $t$ times. Then the generating function for $RD^{\ell,t}$ (See, {{\cite[Th.3]{keith}}}) reads as follows:
	\begin{align*}
		\sum_{n=0}^{\infty} RD^{\ell,t}(n)q^n=
		\frac{E_{\ell}E_t}{E_1E_{\ell t}},
	\end{align*}
	where here and in the sequel, we use the standard $q$-series notation:
 For r$\geq$1, 
	\begin{align*}
		E_r:= \prod_{n=1}^{\infty}(1-q^{nr}), \;\; |q|<1.
	\end{align*}
	For example, $RD^{3,3}(5)=3$, with the corresponding partitions
	\begin{align*}
		(5),\;(4,1),\;(2,2,1).
	\end{align*}
	The OEIS sequence A112194, corresponds to the enumerated function $RD^{3,3}(n)$. Numerous congruence properties are known for the function  $RD^{\ell,t}$. For example, Nadji and Ahmia \cite{Ahmia 1 } proved that:
	For all $n \geq 0$,
	\begin{align*}
	RD^{4,9}(4n+3) &\equiv 0 \pmod{3},\\
	RD^{4,9}(6n+3) &\equiv 0 \pmod{3},\\
	RD^{4,9}(6n+5) &\equiv 0 \pmod{6},\\
	RD^{4,9}(8n+5) &\equiv 0 \pmod{6}.
		\end{align*}
	For more details on $RD^{\ell,t}(n)$, we refer the readers to \cite{keith2,keith,Ahmia 2}.\\
	In this paper, the author initiated a study of the $k$-colored analogue of $RD^{\ell,t}$.
	A partition of a positive integer  is said to be $k$-colored and simultaneously $l$-regular and  $t$-distinct if none of the parts are divisible by $l$, parts appear fewer than $t$ times and appears in $k$ different colors. If $RD_k^{\ell,t}$ denotes such partitions, then its generating function is given by
	\begin{align}\label{Gf}
		\sum_{n=0}^{\infty} RD_k^{\ell,t}(n)q^n=\left(\frac{E_{\ell}E_t}{E_1E_{\ell t}}\right)^k.
	\end{align}
For example, $RD_3^{3,3}(2)=9$, with the relevant partitions being
\begin{align*}
(2_1),(2_2),(2_3),(1_1,1_1),(1_2,1_2),(1_3,1_3),(1_1,1_2),(1_1,1_3),(1_2,1_3).
\end{align*}

The  aim of this paper is to establish several new infinite families of congruence modulo  powers of 3 for $RD_3^{\ell,t}$. Precisely, we focus on the case $(\ell,t)$=(3,3) and (3,27).
 The following are our main results:
	\begin{theorem}\label{T1}
		For integers $k \geq 1$ and $n \geq 0$, we have
		\begin{align*}
			\sum_{n=0}^{\infty}RD_3^{3,3}\left(3^kn+\frac{3^k+1}{2}\right)q^n=\sum_{i=1}^{\frac{3^{k-1}+1}{2}}x_{k,i}q^{i-1}\frac{E_3^{3(4i-2)}}{E_1^{3(4i-2)}}.
		\end{align*}
		where the coefficient vectors
		$	x_{k}=(x_{k,1},\,x_{k,2},\,\ldots)$ are given by
		\begin{align}
			x_1 =(x_{1,1},\,x_{1,2},\,x_{1,3},\,x_{1,4},\,\ldots)
			:= (9,\,0,\,0,\,0,\,\ldots),
			\label{x1}
		\end{align}
		and for $k\geq1$,	
	\begin{align*}
			x_{k+1}=x_k\cdot A,
	\end{align*}
		where 
		$	A:=(a_{i,j})_{i,j\geq1}$ is defined by
	\begin{align}
			\,\,\,\,\,\,\,\,a_{i,j}=m_{4i-2,i+j-1},\label{A}
				\end{align}
				\begin{align*}
					M:=(m_{i,j})_{i,j\ge1}=
	\left(
	\begin{array}{cccccccccc}
		9 & 0 & 0 & 0 & 0 & 0 & 0 & 0 & 0 & \cdots\\
		6 & 243 & 0 & 0 & 0 & 0 & 0 & 0 & 0 & \cdots\\
		1 & 243 & 6561 & 0 & 0 & 0 & 0 & 0 & 0 & \cdots\\
		\vdots & \vdots & \vdots & \vdots & \vdots & \vdots &
		\vdots & \vdots & \vdots & \ddots
	\end{array}
	\right).
		\end{align*}
		\end{theorem}
	Setting $k=1$ and $k=2$ in Theorem \ref{T1}, we obtain
	\begin{equation}\label{I1.2}
	\sum_{n=0}^{\infty}RD_3^{3,3}\left(3n+2\right)q^n=9\frac{E_3^{6}}{E_1^{6}},
	\end{equation}
	
	and
	
	\begin{equation}\label{I1.3}
			\sum_{n=0}^{\infty}RD_3^{3,3}\left(9n+5\right)q^n=
		54\frac{E_3^6}{E_1^6}
		+2187q\frac{E_3^{18}}{E_1^{18}},
		\end{equation}
	respectively.\\
	It is noteworthy that (\ref*{I1.2}) and (\ref*{I1.3}) are analogous to the following two Ramanujan's most beautiful identities \cite[p.~239, p.~243]{Ramanujan}:
	
	\begin{equation*}
		\sum_{n\ge0} p(5n+4)q^n
		=
		5\frac{E_5^5}{E_1^6},
	\end{equation*}
	
	and
	
	\begin{equation*}
		\sum_{n\ge0} p(7n+5)q^n
		=
		7\frac{E_7^3}{E_1^4}
		+49q\frac{E_7^7}{E_1^8},
	\end{equation*}
	respectively.\\
	Furthermore, from Theorem \ref{T1}, we obtain the following corollary: 
	\begin{corollary}\label{coro1}
			For integers $k \geq 1$ and $n \geq 0$, we have
		\begin{align}
				\sum_{n=0}^{\infty}RD_3^{3,3}\left(3^kn+\frac{3^k+1}{2}\right)q^n\equiv0 \pmod{3^{k+1}}.\label{cor1}
		\end{align}
	\end{corollary}
	The result (\ref{cor1}) is analogous to Ramanujan's congruence modulo powers of $5$~\cite{Ramanujan}, for $k,n \geq 0$, 
	\begin{align*}
			p(5^kn+\delta_{5,k})\equiv&\; 0 \pmod{5^k},
	\end{align*}
where $24\delta_{5,k}\equiv1\pmod{5^k}$.
\begin{theorem}For integers $n \geq 0$, we have \label{t3}
	\begin{align} 
		\sum_{n=0}^{\infty}RD_3^{3,27}\left(3n+2\right)q^n=9\left(\frac{E_3}{E_1}\right)^9\left(\frac{E_9}{E_{27}}\right)^3, \label{T3.1}
	\end{align}
	and 
	\begin{align}
		\sum_{n=0}^{\infty}RD_3^{3,27}\left(9n+2\right)q^n=\sum_{i=1}^{3}y_{2,i}q^{i-1}\frac{E_3^{3(4i-2)}}{E_9^{3(4i-2)}}\left(\frac{E_9^3}{E_1^3}\right)^{4i-3}, \label{T3.2}
		\end{align}
			where
			\begin{align*}
					y_2 := (3^2,\,3^7,\,3^{10},\,0,\,0\,\ldots).
			\end{align*}
		\end{theorem}
		Similarly, from Theorem \ref{t3} we obtain the following corollary
\begin{corollary}
	For integers $n \geq 0$, we have
	\begin{align*}
		\sum_{n=0}^{\infty}RD_3^{3,27}\left(3n+2\right)q^n\equiv&0\pmod{3^2}.
			\end{align*}
		
\end{corollary}
\begin{theorem}\label{T5}
	For integers $k \geq 3$ and $n \geq 0$, we have
	\begin{align*}
		\sum_{n=0}^{\infty}RD_3^{3,27}\left(3^kn+\frac{3^k+13}{2}\right)q^n=\sum_{i=1}^{\frac{3^{k-3}\cdot13+1}{2}}y_{k,i}q^{i-1}\frac{E_3^{3(4i-2)}}{E_1^{3(4i-2)}}
	\end{align*}
	where
	\begin{align*}
		y_1
		:=& (9,\,0,\,0,\,0,\,\ldots),\\
		y_2 :=& (3^2,\,3^7,\,3^{10},\,0,\,0\,\ldots),\\
		y_3:=&y_2\cdot B, 
	\end{align*}
	$ where \;\;B:=(b_{i,j})_{i,j\geq1}$ is defined by
		\begin{align*}
			\,\,\,\,\,\,\,\,b_{i,j}=m_{4i-3,i+j-1},
		\end{align*}
	and for $k\geq3$,
	\begin{align*}
		y_{k+1}=y_k\cdot A,
	\end{align*}
	where A is defined as in (\ref{A}).
\end{theorem}
\begin{corollary}\label{coro3}
		For integers $k \geq 3$ and $n \geq 0$, we have
	\begin{align*}
		\sum_{n=0}^{\infty}RD_3^{3,27}\left(3^kn+\frac{3^k+13}{2}\right)q^n\equiv0\pmod{3^{k+1}}.
	\end{align*}
\end{corollary}
	
	
	\noindent Our article is structured as follows: In Section 2, we present some preliminaries required for our proofs. In Section 3 and Section 4, we follow Hirshhhorn's \cite{hirsh}, H operator approach to prove our results. 
	\section{Preliminaries}
\noindent In this section, we first recall the following Jacobi's identity
\begin{equation*}
	E_1^3 = \sum_{n \geq 0} (-1)^n (2n+1) q^{n(n+1)/2}. 
\end{equation*}

\noindent	Which is equivalent to

\begin{align}\label{ak1}
	E_1^3 &= 1 - 3q + 5q^3 - 7q^6 + 9q^{10} - 11q^{15} + 13q^{21} - 15q^{28} + 17q^{36} - 19q^{45} + 21q^{55} -\nonumber \cdots \\ \nonumber
	&= (1 + 5q^3 - 7q^6 - 11q^{15} + 13q^{21} + 17q^{36} - 19q^{45} + \cdots) \\ \nonumber
	&\quad - 3q(1 - 3q^9 + 5q^{27} - 7q^{54} + \cdots) \\ 
	&= P(q^3) - 3q E_9^3,
\end{align}

\noindent	where
\begin{equation*}
	P(q) = \sum_{n=-\infty}^{\infty} (-1)^n (6n+1) q^{n(3n+1)/2}.
\end{equation*}

\noindent	Let $\omega = e^{2\pi i/3}$. Replacing $q$ by $q, \, \omega q$, and $\omega^2 q$ in \eqref{ak1}, and multiplying the three results, we obtain
\begin{align*}
	P(q^3)^3 - 27q^3 E_9^9
	=& E_(q)^3 E(\omega q)^3 E(\omega^2 q)^3\\
	=&\prod_{n=0}^{\infty}(1 - q^n)^3 (1 - \omega^n q^n)^3 (1 - \omega^{2n} q^n)^3\\
=& \prod_{\substack{n=0 \\ 3 \nmid n}} ^{\infty}(1 - q^n)^3 (1 - \omega^n q^n)^3 (1 - \omega^{2n} q^n)^3\\
	&\times \prod_{\substack{n=0\\ 3 \mid n}}^{\infty} (1 - q^n)^3 (1 - \omega^n q^n)^3 (1 - \omega^{2n} q^n)^3 \\
	=&\frac{E_3^{12}}{E_9^3}.
\end{align*}
Therefore, we obtain
\begin{equation*}\label{l21}
	P(q)^3=\frac{E_1^{12}}{E_3^3}+27qE_3^9.
\end{equation*}

\noindent	Let 
\begin{equation}\label{w1}
	\zeta=\frac{E_1^3}{q{E_9}^3}, \ T=\frac{E_3^{12}}{q^3E_9^{12}}.
\end{equation}
That is 
\begin{equation}\label{w2}
	\zeta=\frac{P(q^3)}{qE_9^3}-3.
\end{equation}
It follows that
\begin{equation}\label{w3}
	\zeta^2=\frac{P(q^3)^2}{q^2E_9^6}-6\frac{P(q^3)}{qE_9^3}+9
\end{equation}
and
\begin{equation}\label{w4}
\zeta^3 = \frac{P(q^3)^3}{q^3 E_9^9} - 9 \frac{P(q^3)^2}{q^2 E(q^9)^6}
+ 27 \frac{P(q^3)}{q E_9^3} - 27. 
\end{equation}

\noindent	From \eqref{w1}, \eqref{w2}, \eqref{w3} and \eqref{w4}, we obtain
\begin{equation}\label{a4}
	\zeta^3 + 9\zeta^2 + 27\zeta - T = 0,
\end{equation}

\noindent	which is a modular equation for $\zeta$.  
We now rewrite \eqref{a4} as
\[
\frac{1}{\zeta} = \frac{1}{T}\left(27 + 9\zeta + \zeta^2\right). 
\]

\noindent	It follows that for $i \geq 1$,
\begin{equation}\label{a5}
	\frac{1}{\zeta^i} = \frac{1}{T}\left(\frac{27}{\zeta^{i-1}} + \frac{9}{\zeta^{i-2}} + \frac{1}{\zeta^{i-3}}\right).
\end{equation}

\noindent	Define the Huffing operator $H_3$ as follows
\[
H_3\!\left(\sum_n a(n) q^n\right) = \sum_n a(3n) q^{3n}.
\]

\noindent	Clearly, $H_3(1) = 1$.  
Applying $H_3$ on \eqref{a5}, we obtain
\[
H_3\!\left(\frac{1}{\zeta^i}\right) 
= \frac{1}{T}\left( 27H_3\!\left(\frac{1}{\zeta^{i-1}}\right)
+ 9H_3\!\left(\frac{1}{\zeta^{i-2}}\right)
+ H_3\!\left(\frac{1}{\zeta^{i-3}}\right)\right). 
\]

\noindent	In particular,
\begin{align}
		H_3\!\left(\frac{1}{\zeta}\right) =& \frac{1}{T}\left(27 + 9 \times (-3) + 9\right) 
	= \frac{3^2}{T},\label{1bydelta}\\
	H_3\!\left(\frac{1}{\zeta^2}\right) 
	=& \frac{1}{T}\left(27 \times \frac{9}{T} + 9 \times 1 - 3\right) 
	= \frac{2\cdot 3}{T} + \frac{3^5}{T^2}, \nonumber\\
	H_3\!\left(\frac{1}{\zeta^3}\right) 
	=& \frac{1}{T}\left(27\!\left(\frac{6}{T} + \frac{3^5}{T^2}\right) + 9\!\left(\frac{9}{T}\right) + 1\right) 
	= \frac{1}{T} + \frac{3^5}{T^2} + \frac{3^8}{T^3}.\nonumber
\end{align}
\noindent Now, for $i \geq 1$, we can write
\begin{align}
	H_3\!\left(\frac{1}{\zeta^i}\right)
	= \sum_{j=1}^{i} \frac{m_{i,j}}{T^j}.\label{delta 3}
\end{align}

\noindent	Where $M=(m_{i,j})_{i,j\geq1}$ form a matrix, the first nine rows of which are
\[
M =
\begin{pmatrix}
	3^{2} & 0 & 0 & 0 & 0 & 0 & 0 & \cdots \\
	
	2\cdot 3 & 3^{5} & 0 & 0 & 0 & 0 & 0 & \cdots \\
	
	1 & 3^{5} & 3^{8} & 0 & 0 & 0 & 0 & \cdots \\
	
	0 & 2\cdot 3^{2}\cdot 5 & 2^{2}\cdot 3^{7} & 3^{11} & 0 & 0 & 0 & \cdots \\
	
	0 & 3\cdot 5 & 2^{2}\cdot 3^{5}\cdot 5 & 3^{10}\cdot 5 & 3^{14} & 0 & 0 & \cdots \\
	
	0 & 1 & 2\cdot 3^{6} & 3^{9}\cdot 11 & 2\cdot 3^{14} & 3^{17} & 0 & \cdots \\
	
	0 & 0 & 2^{2}\cdot 3^{2}\cdot 7 & 2\cdot 3^{8}\cdot 7 & 3^{11}\cdot 7^{2} & 3^{16}\cdot 7 & 3^{20} & \cdots \\
	
	0 & 0 & 2^{3}\cdot 3 & 2\cdot 3^{6}\cdot 17 & 2^{4}\cdot 3^{10}\cdot 5 & 2^{2}\cdot 3^{14}\cdot 17 & 2^{3}\cdot 3^{19} & \cdots \\
	
	0 & 0 & 1 & 2\cdot 3^{7} & 3^{10}\cdot 29 & 3^{16}\cdot 5 & 2\cdot 3^{19}\cdot 5 & \cdots \\
	
	\vdots & \vdots & \vdots & \vdots & \vdots & \vdots & \vdots & \ddots
\end{pmatrix}.\]
It is easy to observe, $m_{i,j}=0$, for all $i>3j$ or $j>i$. Hence, we have the following lemma
\begin{lemma}
	For integers $i,j \geq 1$, we have
	\begin{equation}\label{L11}
		H_{3}\!\left(q^{i+1}\frac{E_{3}^{12i-6}}{E_{1}^{12i-6}}\right)=
\sum_{j=1}^{3i-1} a_{i,j}\,
		q^{3j}\frac{E_{9}^{12j-6}}{E_{3}^{12j-6}},
	\end{equation}
	and
	\begin{align}\label{L12}
			H_{3}\!\left(q^{i-2}\frac{E_{3}^{12i-6}}{E_{1}^{12i-6}}\right)=
		\sum_{j=1}^{3i-1} a_{i,j}\,
		q^{3j-3}\frac{E_{9}^{12j-6}}{E_{3}^{12j-6}},
	\end{align}
	where
	\begin{equation*}
		a_{i,j}=m_{4i-2,i+j-1}.
	\end{equation*}
\end{lemma}
\begin{proof}
Note for $1\leq j \leq i-1$, $m_{4i-2,j}=0$, therefore, we have
\noindent
\begin{align*}\label{a15}
	H_3\!\left(\frac{1}{\zeta^{4i-2}}\right)
	= \sum_{j= i}^{4i-2} \frac{m_{4i-2,\,j}}{T^{j}}
	= \sum_{j= 1}^{3i-1} \frac{m_{4i-2,\,i+j-1}}{T^{i+j-1}},
\end{align*}
which can be rewritten as
\begin{equation*}
	H_{3}\!\left(q^{4i-2}\frac{E_{9}^{12i-6}}{E_{1}^{12i-6}}\right)
	=
	\sum_{j=1}^{3i-1}
	m_{4i-2,i+j-1}\,
	q^{3i+3j-3}
	\frac{E_{9}^{12i+12j-12}}{E_{3}^{12i+12j-12}},
\end{equation*}
that gives
\begin{equation*}
H_{3}\!\left(q^{i+1}\frac{E_{3}^{12i-6}}{E_{1}^{12i-6}}\right)
=
\sum_{j=1}^{3i-1} a_{i,j}\,
q^{3j}\frac{E_{9}^{12j-6}}{E_{3}^{12j-6}},
\end{equation*}
where
\[
a_{i,j}=m_{4i-2,i+j-1},
\]
this completes the proof of (\ref{L11}). And (\ref{L12}) follows immediately from identity (\ref{L11}).
\end{proof}
	\section{Proof of Theorems }
	

	\begin{proof}[\it\textbf{Proof of Theorem \ref{T1}}]
			From (\ref{Gf}), we have 
		\begin{align*}
			\sum_{n=0}^{\infty} RD_3^{3,3}(n)q^{n+1}=q\cdot\left(\frac{E_{3}^2}{E_1E_{9}}\right)^3
		\end{align*}
		Applying the $H_3$ operator and then using (\ref{1bydelta}), we have (\ref{I1.2}) which is the k=1 case of Theorem \ref{T1}.\\
		Suppose, Theorem \ref{T1} holds for some $k>1$. Then
			\begin{align*}
			\sum_{n=0}^{\infty}RD_3^{3,3}\left(3^kn+\frac{3^k+1}{2}\right)q^n=\sum_{i=1}^{\frac{3^{k-1}+1}{2}}x_{k,i}q^{i-1}\frac{E_3^{3(4i-2)}}{E_1^{3(4i-2)}},
		\end{align*}
		which is equivalent to
		\begin{align}
			\sum_{n=0}^{\infty}RD_3^{3,3}\left(3^kn+\frac{3^k+1}{2}\right)q^{n+2}=\sum_{i=1}^{\frac{3^{k-1}+1}{2}}x_{k,i}q^{i+1}\frac{E_3^{3(4i-2)}}{E_1^{3(4i-2)}}.\label{3.1}
		\end{align}
		Applying the $H_3$ operator to (\ref{3.1}) and then using (\ref{L11}), we obtain
		\begin{align*}
			\sum_{m=1}^{\infty}RD_3^{3,3}\left(3^k(3m-2)+\frac{3^k+1}{2}\right)q^{3m}=&\sum_{i=1}^{\frac{3^{k-1}+1}{2}}x_{k,i}H_3\left(\ q^{i+1}\frac{E_3^{3(4i-2)}}{E_1^{3(4i-2)}}\right)\\
			=&\sum_{i=1}^{\frac{3^{k-1}+1}{2}}x_{k,i} \sum_{j=1}^{3i-1} a_{i,j}
			q^{3j} \frac{E_9^{3(4j-2)}}{E_3^{3(4j-2)}}\\
			=& \sum_{j=1}^{\frac{3^k+1}{2}}
			\left(\sum_{i=1}^{\frac{3^{k-1}+1}{2}}x_{k,i}a_{i,j}\right)	q^{3j} \frac{E_9^{3(4j-2)}}{E_3^{3(4j-2)}}\\
			=&\sum_{j=1}^{\frac{3^k+1}{2}}x_{k+1,j} q^{3j}\frac{E_9^{3(4j-2)}}{E_3^{3(4j-2)}}.\\
				\end{align*}
		
		We now replace $q^3$ by $q$ and then set $m=n+1$, with $n \geq 0$, in the above identitity to obtain
			\begin{align*}
			\sum_{n=0}^{\infty}RD_3^{3,3}\left(3^{k+1}n+\frac{3^{k+1}+1}{2}\right)q^{n+1}=\sum_{j=1}^{\frac{3^{k}+1}{2}}x_{{k+1},j}q^{j}\frac{E_3^{3(4j-2)}}{E_1^{3(4j-2)}}.	
		\end{align*}
		On rearranging the above, we finally obtain
		\begin{align*}
			\sum_{n=0}^{\infty}RD_3^{3,3}\left(3^{k+1}n+\frac{3^{k+1}+1}{2}\right)q^n=\sum_{j=1}^{\frac{3^{k}+1}{2}}x_{{k+1},j}q^{j-1}\frac{E_3^{3(4j-2)}}{E_1^{3(4j-2)}},
		\end{align*}
		which is $k+1$ case of Theorem \ref{T1}. This completes the proof of Theorem \ref{T1} by induction.
	\end{proof}
	\begin{proof}[\it\textbf{Proof of Theorem \ref{T5}}]
		We have
		\begin{align*}
			\sum_{n=0}^{\infty}RD_3^{3,27}\left(n\right)q^{n+1}=q\left(\frac{E_9}{E_1}\right)^3\left(\frac{{E_3}{E_{27}}}{{E_9}{E_{81}}}\right)^3=\frac{1}{\zeta}\left(\frac{{E_3}{E_{27}}}{{E_9}{E_{81}}}\right)^3
		\end{align*}
		Applying $H_3$ to the above, we find that
		\begin{align*}
			\sum_{m=1}^{\infty}RD_3^{3,27}(3m-1)q^{3m}
			= H_{3}\!\left(\frac{1}{\zeta}\right)\cdot\left(\frac{{E_3}{E_{27}}}{{E_9}{E_{81}}}\right)^3
			= 9q^{3}\frac{E_{9}^{12}}{E_{3}^{12}}\left(\frac{{E_3}{E_{27}}}{{E_9}{E_{81}}}\right)^3.
		\end{align*}
		Since \(m \geq 1\), replacing \(m\) by \(n+1\), with \(n \geq 0\), and then replacing \(q^3\) by \(q\), we have (\ref{T3.1}).
		Again,\\
		\begin{align*}
			\sum_{n=0}^{\infty}RD_3^{3,27}\left(3n+2\right)q^{n+3}=9q^3\left(\frac{E_9}{E_1}\right)^9\left(\frac{{E_3}^9}{{E_9}^6{E_{27}}^3}\right)=9\frac{1}{\zeta^3}\left(\frac{{E_3}^9}{{E_9}^6{E_{27}}^3}\right).
		\end{align*}
			Applying $H_3$ and then substituting (\ref{delta 3}), we have
			\begin{align*}
				\sum_{m=1}^{\infty}RD_3^{3,27}\left(3^2m-7\right)q^{3m}=9\sum_{j=1}^{3} m_{3,j}q^{3j}\left(\frac{E_9}{E_3}\right)^{12j-9}\left(\frac{E_9}{E_{27}}\right)^3.
			\end{align*}
				Since \(m \geq 1\), replacing \(m\) by \(n+1\), with \(n \geq 0\), and then replacing \(q^3\) by \(q\), we have (\ref{T3.2}).
				From identity (\ref{T3.2}), we obtain
			\begin{align*}
				\sum_{n=0}^{\infty}RD_3^{3,27}\left(9n+2\right)q^{n+7}=\sum_{i=1}^{3}y_{2,i}q^{3(3-i)}\frac{E_3^{3(4i-2)}}{E_9^{3(4i-2)}}\left(q\frac{E_9^3}{E_1^3}\right)^{4i-3}.
			\end{align*}	
		Proceeding in the same manner as above, we obtain
			\begin{align*}
				\sum_{n=0}^{\infty}RD_3^{3,27}\left(3^3n+20\right)q^n=\sum_{i=1}^{7}y_{3,i}q^{i-1}\frac{E_3^{3(4i-2)}}{E_1^{3(4i-2)}},
			\end{align*}
	which is the $k=3$ case of Theorem \ref{T5}.
Assume that Theorem \ref{T5} holds for some $k\geq3$. Then
	\begin{align*}
	\sum_{n=0}^{\infty}RD_3^{3,27}\left(3^kn+\frac{3^k+13}{2}\right)q^n=\sum_{i=1}^{\frac{3^{k-3}\cdot13+1}{2}}y_{k,i}q^{i-1}\frac{E_3^{3(4i-2)}}{E_1^{3(4i-2)}},
\end{align*}
which is equivalent to
	\begin{align}\label{n-1}
	\sum_{n=0}^{\infty}RD_3^{3,27}\left(3^kn+\frac{3^k+13}{2}\right)q^{n-1}=\sum_{i=1}^{\frac{3^{k-3}\cdot13+1}{2}}y_{k,i}q^{i-2}\frac{E_3^{3(4i-2)}}{E_1^{3(4i-2)}}.
\end{align}
Applying the $H_3$ operator to (\ref{n-1}) and then applying (\ref{L12}), we obtain
\begin{align*}
	\sum_{m=0}^{\infty}RD_3^{3,27}\left(3^k(3m+1)+\frac{3^k+13}{2}\right)q^{3m}=&\sum_{i=1}^{\frac{3^{k-3}\cdot 13+1}{2}}y_{k,i}H_3\left(\ q^{i-2}\frac{E_3^{3(4i-2)}}{E_1^{3(4i-2)}}\right)\\
=&\sum_{i=1}^{\frac{3^{k-3}\cdot 13+1}{2}}y_{k,i}\sum_{j=1}^{3i-1} a_{i,j}\,
q^{3j-3}\frac{E_{9}^{12j-6}}{E_{3}^{12j-6}}\\
	=&\sum_{j=1}^{\frac{3^{k-2}\cdot13+1}{2}}\left(\sum_{i=1}^{\frac{3^{k-3}\cdot13+1}{2}}y_{k,i}a_{i,j}\right)q^{3j-3}\frac{E_{9}^{12j-6}}{E_{3}^{12j-6}}\\
	=&\sum_{j=1}^{\frac{3^{k-2}\cdot13+1}{2}}y_{{k+1},i}\cdot q^{3j-3}\frac{E_{9}^{3({4j-2})}}{E_{3}^{3({4j-2})}}.
\end{align*}
We now replace $q^3$ by $q$ and then setting $m=n$, with $n \geq 0$, in the above identity to finally obtain
\begin{align*}
	\sum_{n=0}^{\infty}RD_3^{3,27}\left(3^{k+1}n+\frac{3^{k+1}+13}{2}\right)q^{n}=\sum_{j=1}^{\frac{3^{k-2}\cdot13+1}{2}}y_{{k+1},j}q^{j-1}\frac{E_3^{3(4j-2)}}{E_1^{3(4j-2)}},
\end{align*}
which is $k+1$ case of Theorem \ref{T5}. This completes the proof of Theorem \ref{T5} by induction.
	\end{proof}
	\section{Proof of the congruences}
	From \cite[Lemma (3.2)]{tang}, we recall that
	\[\nu(m_{i,j})\ge\left\lfloor\frac{9j-3i-1}{2}\right\rfloor,\]
	and so,
	\[\nu(a_{i,j})=\nu(m_{4i-2,i+j-1})\ge\left\lfloor\frac{9j-3i-4}{2}\right\rfloor.\]
	It is not hard to show that
	\begin{align}
		\nu(x_{{k+1},j})
			&\ge k+2+\left\lfloor\frac{9j-9}{2}\right\rfloor,\label{c1}
	\end{align}
	The identity (\ref{c1}) is true for \(k=0\), by (\ref{x1}).
	Suppose (\ref{c1}) is true for some \(k\ge0\). Then
	\begin{align*}
		\nu(x_{k+2,j})
		&\ge \min_{i\ge1}\left(\nu(x_{k+1,i})+\nu(a_{i,j})\right)\\
		&=\nu(x_{k+1,1})+\nu(a_{1,j})\\
		&\ge k+2+\left\lfloor\frac{9j-7}{2}\right\rfloor\\
			&\ge k+3+\left\lfloor\frac{9j-9}{2}\right\rfloor.
			\end{align*}
	which is the \((k+1)\) case of (\ref{c1}). This completes the proof of (\ref{c1}) by induction.
Proceeding in a similar manner, we can prove that
		\begin{align}
		\nu(y_{{k+1},j})
		&\ge k+2+\left\lfloor\frac{9j-9}{2}\right\rfloor,\label{c2}
	\end{align}
	The Corollary \ref{coro1}  and \ref{coro3} follows from Theorem \ref{T1} and \ref{T5},  respectively, together with (\ref{c1}) and (\ref{c2}).
	
	\bigskip
	
	\bigskip
	
	\noindent\textsuperscript{1,2}Ramanujan School of Mathematical Sciences,\\Department of Mathematics,\\ Pondicherry University,\\ Puducherry - 605014, India. \\
	
	\noindent Email: \texttt{anjelinvallialil@pondiuni.ac.in} \\
	Email: \texttt{dr.fathima.sn@pondiuni.ac.in}  (\Letter)
\end{document}